\documentclass[12pt,a4paper]{article}

\usepackage{mathrsfs}
\usepackage{amsfonts}
\usepackage{amssymb}
\usepackage{amsmath}
\newtheorem{theorem}{Theorem}[section]

\numberwithin{equation}{section}

\newtheorem{definition}[theorem]{Definition}
\newtheorem{example}[theorem]{Example}

\newtheorem{lemma}[theorem]{Lemma}

\newtheorem{remark}[theorem]{Remark}

\newenvironment{proof}[1][Proof]{\textbf{#1.}}{\ \rule{0.5em}{0.5em}}%

\begin{document}
\parindent 9mm
\title{Perturbations of Admissibility, Exact Controllability,
Exact Observability and Regularity
\thanks{This work was supported by the National Natural Science Foundation of China (grant nos. 11301412 and 11131006), Research Fund for the Doctoral Program of Higher Education of China (grant no. 20130201120053), Natural Science Foundation of Shaanxi Province (grant no. 2014JQ1017), Project funded by China Postdoctoral Science Foundation (grant nos. 2014M550482 and 2015T81011). Part of this work was done during the first author was visiting Prof. Bao-Zhu Guo at
Academy of Mathematics and Systems Science, The Chinese Academy of Sciences.}
}
\author{ Zhan-Dong Mei
\thanks{Corresponding
author, School of Mathematics and Statistics, Xi'an Jiaotong
University,
 Xi'an
710049, China; Email: zhdmei@mail.xjtu.edu.cn } \ \ \
 Ji-Gen Peng \thanks{School of Mathematics and Statistics, Xi'an Jiaotong
University,
 Xi'an
710049, China; Email: jgpeng@mail.xjtu.edu.cn}}


\date{}
\maketitle \thispagestyle{empty}
\begin{abstract}
%
%
This paper is concerned with the notions of admissibility, exact controllability, exact observability and regularity of linear systems in the Banach space setting. It is proved that admissible controllability, exact controllability, admissible observation, exact observability and regularity are invariant under some regular perturbations of the generators, such results are generalizations of some previous references. Moreover, the related boundary linear systems and some illustrative examples are presented.

\vspace{0.5cm} 

%
%

\end{abstract}


\section{Introduction}
In the theory of finite dimensional linear control system,
the final state and output are continuously depended on the initial state and input. Observe that such
continuous dependence is the essential property in system theory. Motivated by this, Salamon \cite{Salamon1987} introduce the class of well-posed liner systems by continuous dependence in Hilbert space setting. Later, Weiss \cite{Weiss1989a,Weiss1989b,Weiss1989c} simplified Salamon's
theory; he described well-posed linear system equivalently by using four algebraic equations (see the description in Section 2). In the functional analysis frame, the control operators and observation operators of well-posed linear system may be unbounded, which allow ones to study partial differential equations with boundary control and boundary observation. Over the last decades there has been a growing interest in well-posedness of partial differential equations with control and observation on the boundary, and it has been proved that
many partial differential equations can be formulated as well-posed linear systems
\cite{Ammari2002,Guo2009Z,Lasiecka1992,Lasiecka2003,Malinen2006,Salamon1987,Salamon1989}.
Regular linear systems, introduced by Weiss \cite{Weiss1989c}, are among the well-posed systems whose output function corresponding to a step input function and zero initial state is not very discontinuous at zero (see the definition in Section
2). Many well-posed physical systems are also regular, see \cite{Chai2010a,Guo2005S,Guo2006S,Guo2012S,Guo2005Z,Guo2007Z,
Hadd2005b,Hadd2006b,Weiss2002}. Regular linear systems have a convenient representation, similar to that of finite dimensional systems. Concretely, Weiss showed in \cite{Weiss1989c} that regular
linear systems with unbounded control and observation operators
can be simply represented by
\begin{eqnarray*}
  \dot{x}(t)=Ax(t)+Bu(t), y(t)=C_\Lambda x(t)+Du(t),
\end{eqnarray*}
where $C_\Lambda$ is the $\Lambda$-extension of the observation
operator $C$ with respect to system operator $A$ (see Section 2). In this sense,  an infinite-dimensional regular linear systems have the characteristic of ``finite-dimensional systems". Many references were concerned with abstract control theory under the frame of regular linear systems.

In order to obtain a well-posed and regular linear system, the control and observation operators should be admissible
for the system operator (see, e.g.,
\cite{Salamon1987,Salamon1989,Weiss1989c,Weiss1994a,Weiss1994b}).
Hence the the concepts of admissibility of control and
observation operators have been discussed by many references,
most of which are interested in proving or
disproving Weiss' conjecture (see, e.g.,
\cite{Haak2006,Jacob2001a,Zwart2003,Zwart2005}).
Here, we mention an important work due to Zwart \cite{Zwart2005}; he
proved that the Weiss conjecture almost holds in Hilbert spaces.

For admissible control and admissible observation system, one can consider the notion of exact controllability and exact observability, because the enters into the study of many other important concepts. For instance, exact controllability is closely related to stabilizability and optimizability, while exact controllability is closely related to detectability and estimatability \cite{Guo2002L,Lasiecka2003,Weiss2000}. Exact controllability and exact observability have received considerable attention in the functional analysis frame (see e.g. \cite{Haak2012,Jacob2001,Jacob2009,Lasiecka1992a,Partington2005,Rebarber2000,
Russell1994,Xu2008}), where some necessary and/or sufficient conditions have been given.

Generally, it is not an easy task to verify the admissibility, exact controllability, exact observability and regularity for a specific linear system with boundary control and /or boundary observation. Due to the difficulties of direct proving the admissibility and regularity, perturbation method has been successfully used to study the the admissibility and regularity. Weiss \cite{Weiss1989b} discussed the admissibility
of observation system under bounded perturbation of the system operator, namely, $C$ being admissible observable operator for $A$ implies that $C$ is admissible for $A+P$, provided $P$ is an bounded linear operator on the state space.
In \cite{Weiss1994b}, Weiss showed that the closed-loop system of well-posed linear system preserves the admissibility, exact controllability and exact observability. Moreover, the closed-loop system of regular linear system preserve the regularity.
Hadd \cite{Hadd2005a} proved that both $B$ and $\Delta A$ are $p$-admissible controllable operators for $A$ imply that $B$ is $p$-admissible for $(A+\Delta A)|_X$; $((A+\Delta A)|_X,B)$ is exactly controllable provided $(A,B)$ is exactly $p$-controllable and $\Delta A$ is ``small" enough. In their paper \cite{Hadd2006}, Hadd showed that $C$ and $\Delta A$ being $p$-admissible observable operators for $A$ implies $C$ is $p$-admissible for $ A+\Delta A$. Moreover, Tucsnak and Weiss \cite{Tucsnak2009} proved that if $(A,C)$ is exactly observable and $\Delta A$ is ``small" enough, then $(A+\Delta A,C)$ is exactly observable. Later, Mei and Peng \cite{Mei2010b,Mei2014} weakened the condition of \cite{Hadd2005a,Hadd2006,Tucsnak2009} that $\Delta A$ is $p$-admissible controllable (observable) operator to $\Delta A$ being $q$-admissible controllable (observable) operator. Mei and Peng \cite{Mei2010} proved that the admissibility, exact controllability and exact observation are preserved under cross perturbations, that is, $(A,B,\Delta A)$ is a regular linear system, then $B$ is admissible for $A+\Delta A$ and $(A+\Delta A,B)$ is exactly controllable provided $(A,B)$ is exactly controllable and $\Delta A$ is ``small" enough; $(A,\Delta A,C)$ is a regular linear system, then $C_\Lambda^A$ is admissible for $(A_{-1}+\Delta A)|_X$ and $((A_{-1}+\Delta A)|_X,C_\Lambda^A)$ is exactly observable provided $(A,C)$ is exactly observable and $\Delta A$ is ``small" enough. In their paper \cite{Mei2010a}, Mei and Peng proved that $(A,B,\Delta A)$ and $(A,B,C)$ generating regular linear systems imply that $(A+\Delta A,B,C_\Lambda^A)$ generates a regular linear system; $(A,\Delta A,C)$ and $(A,B,C)$ generating regular linear systems imply that $((A_{-1}+\Delta A)|_X,B,C)$ generates a regular linear system.

The aim of this paper is to study some general perturbation theorems of admissibilities, exact controllabilities, exact observations and regularities. Apart from the introduction, our arrangement is as follows.  In Section 2, we introduce  some basic notions and properties related to regular linear systems and boundary systems; the notions of exact controllability and exact observation are also be introduced. Section 3 is to
give our main results. Concretely, we obtain admissible controllability, admissible observation, exact controllability, exact observation and regularity under some regular perturbations. Moreover, all the perturbation results are used to solve the corresponding boundary systems. The systems governed by  specific partial differential equations are presented to illustrate our results.

\section{Preliminaries}

In this section, we recall some definitions related to
regular linear systems and boundary linear systems.
As stated in the introduction, Weiss has showed that the continuous dependence of state and output on the initial state and input can be simplified to four algebraic equations. We adopted Weiss' definition for well-posed linear system \cite{Weiss1989c}.

\begin{definition}\label{wellposed}
A quadruple $\Sigma=(T,\Phi,\Psi,F)$ is said to be a well-posed linear system on $(X,U,Y)$, if the following four
conditions are satisfied:

(i) $T=\{T(t)\}_{t\geq 0}$ is a
$C_0$-semigroup generated by $A$ on $X$;

(ii) $\Phi=\{\Phi(t)\}_{t\geq
0}$ is a family of bounded linear operators, called {\it input maps}, from $L^p(R^+,U)$ to $X$ such that
$$\Phi(t+\tau)u=T(t)\Phi(\tau)u+\Phi(t)u(\cdot+\tau), \ \forall u\in
L^p(R^+,U),\tau\geq 0, t\geq 0,$$
we call $(T,\Phi)$ an {\it
    abstract linear control system;}

(iii) $\Psi=\{\Psi(t)\}_{t\geq
0}$ is a family of bounded linear operators, called {\it output maps}, from
$X$ to $L^p(R^+,Y)$ such that
$$(\Psi(t+\tau) x)(s)=(\Psi(t)
T(\tau) x)(s-\tau),\ \forall x\in X, t+\tau\geq s\geq\tau\geq 0,t\geq 0,$$
we call $(T,\Psi)$ an {\it
    abstract linear observation system;}

(iv) $F=\{F(t)\}_{t\geq
0}$ is a family of bounded linear operators, called {\it
input-output map}, from $L^p(R^+,U)$ to $L^p(R^+,Y)$ such that
$$(F(t+\tau) u)(s)=(\Psi(t) \Phi(\tau)u+F(t)
u(\cdot+\tau))(s-\tau),\ \forall u\in L^p(R^+,U),t+\tau\geq s\geq\tau\geq 0,t\geq 0.$$
\end{definition}

By a representation theorem due to Salamon \cite{Salamon1989} (see also Weiss \cite{Weiss1989a}), corresponding to abstract linear control system $(T,\Phi)$, there is a unique control operator $B\in L(U,X_{-1})$, called {\it admissible control operator} (also $p$-admissible control operator), satisfying
$$\Phi(t)u=\int_0^tT_{-1}(t-s)Bu(s)ds\in X, \forall u\in
L^p(R^+,U),t\geq 0.$$
Here $T_{-1}$ is the extrapolation
semigroup, which is the continuous extension of $T$ to the
extrapolation space $X^A_{-1}$ defined by the completion
of $X$ under the norm $\|R(\lambda_0,A)\cdot\|$ with
$R(\lambda_0,A)$ being the resolvent of $A$ and $\lambda$ belonging the resolvent set of $A$. The generator of $T_{-1}$ is the continuous extension of $A$ to $X$ and is denoted by $A_{-1}$.
In this case, we also say $(A,B)$ generates an abstract
linear control system and denote $\Phi_{A,B}$.
Moreover, if $\Phi_{A,B}(\tau)$ is surjective, we call $(A,B)$ to be exactly controllable (also exactly $p$-controllable) at $\tau.$

It follows from Salamon \cite{Salamon1989} or Weiss \cite{Weiss1989b} that an abstract linear observation system $(T,\Psi)$ corresponds a unique operator, called {\it admissible observation operator} (also $p$-admissible observation operator) $C\in L(D(A),Y)$ satisfying
$$\int_0^{t_0}\|CT(t)x\|^pdt\leq c(t_0)\|x\|^p, \forall x\in
D(A)$$ such that $(\Psi(t) x)(\tau)=CT(\tau)x, \forall x\in
D(A),\tau\leq t$.
In this case, we also say $(A,C)$ generates an abstract
linear control system and denote $\Psi_{A,C}$. Moreover, $(A,C)$ is called to be exactly observable (also exactly $p$-observable) at $\tau$, provided there exists a constant $k>0$ such that $\|\Psi_{A,C}(\tau)x\|\geq k\|x\|,\ x\in X$.

Let $\Sigma=(T,\Phi,\Psi,F)$ be well-posed linear system.  For any $x(0)\in X, u\in L^p_{loc}(R^+,U)$, $x(t)=T(t)x(0)+\Phi(t)u$ is the solution of equation $\dot{x}(t)=A_{-1}x(t)+Bu(t).$ Define output $y=\Psi(\infty) x_0+F(\infty) u,$
 where $\Psi(\infty):\hbox{ }X\rightarrow L^2_{loc}(R^+,Y)$ and $F(\infty):\hbox{ }L^p_{loc}(R^+,U)\rightarrow L^p_{loc}(R^+,Y)$ are the extended output map defined by the
strong limit of $\Phi(\tau)$ and $F(\tau)$ as $\tau\rightarrow +\infty$, respectively (see \cite{Weiss1989b,Weiss1989c}).
In the special case $u=0$, it follows from \cite[Theorem 4.5 and Proposition 4.7]{Weiss1989b} that for any $x(0)\in X$, $y(t)=C_\Lambda^A T(t)x(0)$ a.e. $t\geq 0$, where $C_\Lambda^A$ defined by
\begin{equation}
C_\Lambda^A x=\lim_{\lambda\rightarrow \infty}C\lambda R(\lambda,A)x,\ x\in D(C_\Lambda^A)=\{x\in X:\mbox{this above limit exists in}\, Y\}
\end{equation}
 is called {\it $\Lambda$-extension} of $C$ with respect to $A$. By \cite{Staffans2002}, it follows that the output $y(t)$ can be expressed by
$$y(t)=C_\Lambda^A[x(t)-(\lambda-A_{-1})^{-1}B]u(t)+G(\lambda)u(t),$$
$a.e.\ t\geq 0,$ where $G(\lambda)$ is the transform function.
It is not hard to see that Definition \ref{wellposed} implies continuous dependence, that is, there exist positive function $m$ and $n$ on $R^+$ such that
$$\|x(t)\|+\|y\|_{L^{p}([0,t],Y)}\leq m(t)\|x(0)\|+n(t)\|u\|_{L([0,t],U)},\ t\geq 0.$$

The well-posed linear system $\Sigma$ is called a
{\it regular linear system} if, there exists a bounded operator $D$, called {\it feedthrough operator}, such that the limit $$\lim\limits_{s\rightarrow
0}\frac{1}{s}\int_0^s(F(t) u_0)(\sigma)d\sigma=Dz$$
exists in $Y$
for the constant input $u_0(t)=z$, $z\in U$, $t\geq 0$.
Weiss showed that well-posed linear system $\Sigma$ is regular if and only if $G(\lambda)$ strongly converges to $D$ as
$\lambda\rightarrow +\infty$, that is,
$$\lim_{\lambda\rightarrow +\infty}G(\lambda)u=Du,\ u\in U.$$
The regular linear system is described by
\begin{eqnarray*}
  \dot{x}(t)=Ax(t)+Bu(t), y(t)=C_\Lambda^A x(t)+Du(t).
\end{eqnarray*}

\begin{definition}\cite{Staffans2005,Weiss1994b}
An operator $\Gamma\in L(Y,U)$ is called an admissible feedback operator
for $\Sigma=(T,\Phi,\Psi,F)$ if $I-F(t)\Gamma$ is invertible for some $t\geq 0$ (hence any $t\geq 0$).
\end{definition}

\begin{theorem}\label{relation}\cite{Weiss1994b}
Let $(A,B,C,D)$ be the generator of regular linear system
$\Sigma=(T,\Phi,\Psi,F)$ on $(X,U,Y)$ with admissible feedback
operator $\Gamma\in L(Y,U)$. Suppose that $I-D\Gamma$ is invertible. Then the feedback system $\Sigma^\Gamma$ is
a well-posed linear system generated by
$(A^\Gamma,B^\Gamma,C^\Gamma)$:
\begin{eqnarray*}
  &A^\Gamma =(A_{-1}+B\Gamma(I-D\Gamma)_{left}^{-1} C^A_\Lambda)|_X,\;\\
 &D(A^\Gamma): = \{z\in D(C^A_\Lambda):(A_{-1}+B\Gamma (I-D\Gamma)^{-1}C^A_\Lambda)z\in
 X\},
\end{eqnarray*}
$C^\Gamma=(I-D\Gamma)^{-1}C^A_\Lambda$ restricted to
$D(A^\Gamma)$ and $B^\Gamma=J^{A,A^\Gamma}(I-D\Gamma)^{-1}B$, where $J^{A,A^\Gamma}$ is defined by
$J^{A,A^\Gamma}x=\lim_{\lambda\rightarrow
\infty}(\lambda-A_{-1})^{-1}x$ (in $X_{-1}^{A^\Gamma}$) with
$D(J^{A,A^\Gamma})=\{x\in X^A_{-1}:$ the limit
$\lim_{\lambda\rightarrow \infty}(\lambda-A_{-1})^{-1}x$ exists
$\}$.
\end{theorem}

In the rest of this section, we introduce some notions related to linear boundary system described in the abstract frame as follows
\cite{Malinen2006,Salamon1987}.
 \begin{eqnarray}\label{boundctr}
 \left\{
   \begin{array}{ll}
     \dot{z}(t)&=Lz(t), \\
     Gz(t)&=u(t), \\
     y(t)&=Kz(t),
   \end{array}
 \right.
 \end{eqnarray}
 where $\bigg[\begin{array}{c}
            L \\
            G \\
            K
          \end{array}\bigg]
   $ is closed linear operators from $D(L)$ to space $X \times U\times Y$; $D(L)$ is continuously embedded in $X$;
  $G$ is surjection and $Ker\{G\}:=\{z\in Z: Gz=0\}$ is dense in $X$;
  $L|_{Ker\{G\}}$ generates a $C_0$-semigroup on $X$.
We denote system (\ref{boundctr}) by $(L,G,K)$ for brief.

Denote $A=L|_G$, $C=K|_{D(A)}$. By \cite{Greiner1987}, $D(L)$ can be
decomposed to direct sum $D(L)=D(A)\bigoplus Ker\{\lambda-L\}$ and
the operator $G$ is bijective from $Ker\{\lambda-L\}$ onto $U$,
where $\lambda$ is any component of resolvent set $\rho(A)$ of $A$. Hence we
can denote $D_{\lambda,L,G}$ by the solution operator from $z$ to
$u$ of the following function
\begin{align*}
    \left\{
      \begin{array}{ll}
        (\lambda-L)z=0, & \hbox{} \\
        Gz=u, & \hbox{}
      \end{array}
    \right.
\end{align*}
that is $z=D_{\lambda,L,G}u.$ By \cite{Malinen2006,Salamon1987}, it follows that boundary control system
\begin{align*}
    \left\{
      \begin{array}{ll}
        \dot{z}(t)=Lz(t), & \hbox{} \\
        Gz(t)=u(t), & \hbox{}
      \end{array}
    \right.
\end{align*}
is equivalent to system
$$\dot{z}(t)=A_{-1}z(t)+Bu(t)$$
in the sense of classical solution, where $B$ is given by $   B=(\lambda-A_{-1})D_{\lambda,L,G}\in L(U,X_{-1})$.
Concretely, $z(t)$ and $u(t)$ satisfying $Gz(t)=u(t)$ is equivalent to $A_{-1}z(t)+Bu(t)\in X$; if $z(0)\in X$, $u\in W^{2,p}(R^+,U)$ satisfying $A_{-1}z(0)+Bu(0)\in X$, we have
$\dot{z}(t)=Lz(t)=A_{-1}x(t)+Bu(t)$. Moreover the initial condition implies that $y(t)=
Kz(t)=C(x(t)-(\lambda-A_{-1})^{-1}Bu(t))+K(\lambda-A_{-1})^{-1}Bu(t).$
Hence $C$ is the observation operator of the boundary system $(L,G,K)$ and the corresponding transform function is $K(\lambda-A_{-1})^{-1}B,\ \lambda\in \rho(A)$.

Boundary system $(L,G,K)$ is well-posed if
there exist positive function $m$ and $n$ on $R^+$ such that
$$\|x(t)\|+\|y\|_{L^{p}([0,t],Y)}\leq m(t)\|x(0)\|+n(t)\|u\|_{L([0,t],U)},\ t\geq 0.$$
It is regular
if it is well-posed and the strong limit of the transform function exists, that is, $\lim_{\lambda\rightarrow
+\infty}KD_{\lambda,L,G} u$ exists for any $u\in U$. In this case, denote by $\overline{K}_{A,B}$ the corresponding feedthrough operator, which means $\overline{K}_{A,B}u:=\lim_{\lambda\rightarrow +\infty}D_{\lambda,L,G} u,\ u\in U$. Then boundary
system $(L,G,K)$ is regular with generator $(A,B,C,\overline{K}_{A,B})$. We also say that the generator is $(A,B,K,\overline{K}_{A,B})$ and denote $K_\Lambda^A$ by $C_\Lambda^A=(K|_{D(A)})_\Lambda^A$.

The following two lemmas will be used in the next section.
\begin{lemma}\label{1output}\cite{Mei2015}
Assume that the boundary control system
\begin{align*}
  (BCS)\left\{
   \begin{array}{ll}
   \dot{z}(t)&=Lz(t)\\
     Gz(t)&=u(t)
   \end{array}
 \right.
\end{align*}
is an abstract linear control system generated by
$(\mathbb{A},\mathbb{B})$. Then the boundary system $(L,G,Q)$ is a
regular linear system on $(X,U,Y)$ if and only if
$(\mathbb{A},\mathbb{B},Q)$ generates a regular linear system. In
this case, for any $z\in Z$, we have
$$Qz=Q_\Lambda^\mathbb{A} z+\overline{Q} Gz.$$
\end{lemma}

\begin{lemma}\label{1boundarycontrol}\cite{Mei2015}
Assume that the boundary system $(L,G,Q)$ is a regular linear system generated by $(A,B,Q,\overline{Q}_{A,B})$ on $(X,U,X)$ with admissible feedback operator $I$. Then the system
\begin{align*}
  (OS)\left\{
   \begin{array}{ll}
   \dot{z}(t)&=Lz(t)\\
     Gz(t)&=Qz(t)+v(t)\\
   \end{array}
 \right.
\end{align*}
is an abstract linear control system generated by
$(A^I,B^I).$
\end{lemma}

\section{Main Results}
In this section, we shall obtain the admissible controllability, exact controllability, admissible observation, exact observation and regularity under some regular perturbations. Moreover, all the perturbation results are used to solve the corresponding boundary systems. The systems governed by specific partial differential equations are presented after every perturbation result to illustrate our results.

We first consider the admissible controllability and exactly
controllability under associated perturbation.
To prove the robustness of exact controllability, we introduce the following important lemma related to radius of surjectivity.
\begin{lemma}\label{Kolmogorov1980}\cite[page 227]{Kolmogorov1980}
Let $E$ and $F$ be Banach spaces. Then, $\mathfrak{S}(E,F):=\{\Xi\in L(E,F): \hbox{ }\Xi \hbox{ }is \hbox{ }
  surjective\}$ is an
open set in $L(E,F)$, i.e., given $\Pi\in\mathfrak{S}(E,F)$, there
exists $\alpha>0$ such that
\begin{eqnarray*}
  \{\Xi \in L(E,F):\|\Pi-\Xi\|<\alpha\}\subset
  \mathfrak{S}(E,F).
\end{eqnarray*}
The constant $\alpha$ is called {\it radius of surjectivity} of
$\Pi$.
\end{lemma}

\begin{theorem}\label{across}
Let $(A,B,C,D)$ generate a regular linear system with admissible feedback operator $I$ on $(X,Y,Y)$, $I-D$ is invertible, and $(A,\Delta B,C,P)$ generate a regular linear system on $(X,U,Y)$. Then $(A^I,J^{A,A^I}B(I-D)^{-1}P+J^{A,A^I}\Delta B,(I-D)^{-1}C_\Lambda^A)$ generates a regular linear system,
and there holds
$$\Phi_{A^I,J^{A,A^I}B(I-D)^{-1}P+J^{A,A^I}\Delta B}=\Phi_{A,B}(I-F_{A,B,C,D})^{-1}F_{A,\Delta B,C,P}+\Phi_{A,\Delta B},$$
where $A^I=(A+B(I-D)^{-1}C_\Lambda^A)|_X.$
Moreover, if $(A,\Delta B)$ is exactly controllable at $t_0>0$,
then there exists $k_0>0$ such that $(A^I(k),J^{A,A^I(k)}B(I-kD)^{-1}kP+J^{A,A^I(k)}\Delta B)$ is also exactly controllable at $t_0$
whenever $k<k_0,$ where $A^I(k)=(A+B(I-kD)^{-1}kC_\Lambda^A)|_X.$
\end{theorem}
\begin{proof}\ \
We consider the
operators $\tilde{B}:=(B, \Delta B): Y\times U\rightarrow X_{-1}$,
$\tilde{C}=\left(
                                                          \begin{array}{c}
                                                            C \\
                                                            0 \\
                                                          \end{array}
                                                        \right)
: X_1\rightarrow Y\times U$, $\tilde{D}=\left(
                                  \begin{array}{cc}
                                    D & P \\
                                    0 & 0 \\
                                  \end{array}
                                \right):Y\times U\rightarrow
                                Y\times U$.

Since $(A,B,C,D)$ and $(A,\Delta B,C,P)$ generate regular linear systems, it is easy to verify that
$(A,\tilde{B},\tilde{C},\tilde{D})$ generates a regular linear system given by
\begin{eqnarray*}
  \Sigma_{A,\tilde{B},\tilde{C}}:=\left(
                                    \begin{array}{cc}
                                      T & (\Phi_{A,B},\Phi_{A,\Delta B}) \\
                                      \left(
                       \begin{array}{c}
                         \Psi_{A,C} \\
                         0 \\
                       \end{array}
                     \right)& \left(
                                 \begin{array}{cc}
                                   F_{A,B,C,D} & F_{A,\Delta B,C,P} \\
                                   0 & 0 \\
                                 \end{array}
                               \right)
                      \\
                                    \end{array}
                                  \right).
\end{eqnarray*}
Observe that $I$ is an admissible feedback operator for $\Sigma_{A,B,C,D}$. We have that $$I_{Y\times U}-\left(
                                       \begin{array}{cc}
                                         F_{A,B,C,D} & F_{A,\Delta B,C,P} \\
                                         0 & 0 \\
                                       \end{array}
                                     \right)
=\left(
                                       \begin{array}{cc}
                                         I-F_{A,B,C,D} & -F_{A,\Delta B,C,P} \\
                                          0& I \\
                                       \end{array}
                                     \right)$$ is invertible
and
$$\bigg(I_{Y\times U}-\left(
                                       \begin{array}{cc}
                                         F_{A,B,C,D} &F_{A,\Delta B,C,P} \\
                                         0  & 0 \\
                                       \end{array}
                                     \right)\bigg)^{-1}=\left(
                                       \begin{array}{cc}
                                         (I-F_{A,B,C,D})^{-1} & (I-F_{A,B,C,D})^{-1}F_{A,\Delta B,C,P} \\
                                         0 & I \\
                                       \end{array}
                                     \right),$$
that is, $I_{X\times U}$ is an admissible feedback operator for $\Sigma_{A,\tilde{B},\tilde{C},\tilde{D}}$. It follows from Theorem \ref{relation} that $A^{I_{X\times U}}=(A_{-1}+\tilde{B}(I-\tilde{D})\tilde{C}_\Lambda^{\tilde{A}})|_X
=(A_{-1}+B(I-D)^{-1}C_\Lambda^A)|_X=A^I$,
$\tilde{B}^{I_{X\times U}}=J^{A,A^I}\tilde{B}(I-\tilde{D})^{-1}=\left(
                                                \begin{array}{cc}
                                                  J^{A,A^I}B(I-D)^{-1} & J^{A,A^I}B(I-D)^{-1}P+J^{A,A^I}\Delta B \\
                                                \end{array}
                                              \right)
$, $\tilde{C}^I_{X\times U}=(I-\tilde{D})^{-1}\tilde{C}_\Lambda^{A}=
\left(
  \begin{array}{c}
    (I-D)^{-1}C_\Lambda^A \\
    0 \\
  \end{array}
\right)
$
and
\begin{align*}
 &\Phi_{A^I,J^{A,A^I}B(I-D)^{-1}P+J^{A,A^I}\Delta B}\\
=&\Phi_{A^I,\tilde{B}^{I_{X\times U}}}\left(
                                       \begin{array}{c}
                                         0 \\
                                         I \\
                                       \end{array}
                                     \right)\\
=&\Phi_{A,\tilde{B}}(I-F_{A,\tilde{B},\tilde{C},\tilde{D}})^{-1}
\left(
                                       \begin{array}{c}
                                         0 \\
                                         I \\
                                       \end{array}
                                     \right)\\
=& (\Phi_{A,B},\Phi_{A,\Delta B})
\left(
                                       \begin{array}{cc}
                                         (I-F_{A,B,C,D})^{-1} & (I-F_{A,B,C,D})^{-1}F_{A,\Delta B,C,P} \\
                                         0 & I \\
                                       \end{array}
                                     \right)
\left(
                                       \begin{array}{c}
                                         0 \\
                                         I \\
                                       \end{array}
                                     \right)\\
=&\Phi_{A,B}(I-F_{A,B,C,D})^{-1}F_{A,\Delta B,C,P}+\Phi_{A,\Delta B}.
\end{align*}
Moreover, it is not hard to see that
$$(A^I,J^{A,A^I}B(I-D)^{-1}P+J^{A,A^I}\Delta B,(I-D)^{-1}C_\Lambda^A)=\bigg(A^I,\tilde{B}^{I_{X\times U}}\left(
                                                            \begin{array}{c}
                                                              0 \\
                                                              I \\
                                                            \end{array}
                                                          \right)
,\left(
   \begin{array}{c}
     I \\
     0 \\
   \end{array}
 \right)
\tilde{C}^{I_{X\times U}})
\bigg)
$$ is a regular linear system.

Below we prove the robustness of exact controllability.
Observe that $I-kD$ is invertible for any $k<\frac{1}{\|D\|}$ (if $\|D\|=0$, $\frac{1}{\|D\|}=+\infty$).
By \cite[Proposition 3.12 and Proposition 4.10]{Weiss1994b}, $kI$ is admissible feedback for $(A,B,C,D)$ whenever for $k<\frac{1}{\|D\|}$, which indicates that $I$ is admissible feedback for $(A,B,kC,kD)$.
Since $(A,\Delta B)$ is exactly controllable at $t_0>0$,
$\Phi_{A,\Delta B}(t_0)$ is surjective. Let $s_0$ be the radius of surjectivity of $\Phi_{A,\Delta B}(t_0)$. It follows from the above proof that
\begin{align*}
    &\|\Phi_{A^I,J^{A,A^I(k)}B(I-kD)^{-1}kP+J^{A,A^I(k)}\Delta B}(t_0)-\Phi_{A,\Delta B}(t_0)\|\\
=&\|\Phi_{A,B}(t_0)(I-F_{A,B,kC,kD}(t_0))^{-1}F_{A,\Delta B,kC,kP}(t_0)\|\\
\leq &k\|\Phi_{A,B}(t_0)(I-kF_{A,B,C,D}(t_0))^{-1}\|\|F_{A,\Delta B,C,P}(t_0)\|.
\end{align*}
Let $$k_0=\min\{\frac{1}{\|D\|},\frac{1}{\|F_{A,B,C,D}(t_0)\|},
\frac{s_0}{\|\Phi_{A,B}\|\|F_{A,B,C,P}(t_0)\|+s_0\|F_{A,B,C,D}(t_0)\|} \}.$$
Then $\|\Phi_{A^I,J^{A,A^I(k)}B(I-kD)^{-1}kP+J^{A,A^I(k)}\Delta B}(t_0)-\Phi_{A,\Delta B}(t_0)\|<s_0$ whenever $k<k_0$. It follows from Lemma \ref{Kolmogorov1980} that $\Phi_{A^I,J^{A,A^I(k)}B(I-kD)^{-1}kP+J^{A,A^I(k)}\Delta B}(t_0)$ is surjective. This implies that  $(A^I(k),J^{A,A^I(k)}B(I-kD)^{-1}kP+J^{A,A^I(k)}\Delta B)$ is exactly controllability at $t_0$. The proof is therefore completed.
\end{proof}
\begin{remark}
In the special case that $Y=X,$ $D=0,$ $P=0$ and $C=I$, the above theorem says that both $B$ and $\Delta B$ being admissible for $A$ implies that $J^{A,(A_{-1}+B)|_X}\Delta B$ is admissible for $(A_{-1}+B)|_X$, such result has been proved by Hadd \cite{Hadd2006}, as mentioned in the introduction section.
If $Y=X,$ $D=0,$ $P=0$ and $B=I$, Theorem \ref{across} tells that $(A,\Delta B,C)$ generating a regular linear system implies that $(A+C,J^{A,A+C}\Delta B,C)$ generates a regular linear system, particularly, $J^{A,A+C}\Delta B$
is admissible for $A+C$, such result has been proved by Mei and Peng \cite{Mei2010a}.
This means that our result is a generalization of \cite{Hadd2006} and \cite{Mei2010a}.
\end{remark}

\begin{theorem}\label{fedin}
Assume that the boundary system
$  \left\{
   \begin{array}{ll}
   \dot{z}(t)&=Lz(t)\\
     G_1z(t)&=u(t)\\
     G_2z(t)&=0\\
     y(t)&=Kz(t)
   \end{array}
 \right.$
is a regular linear system generated by $(A,B_1,K,\overline{K}_{A,B_1})$ on $(X,U,U)$ with $I$ being admissible feedback operator. Suppose that
 $  \left\{
   \begin{array}{ll}
   \dot{z}(t)&=Lz(t)\\
     G_1z(t)&=0\\
     G_2z(t)&=v(t)\\
     y(t)&=Kz(t)
   \end{array}
 \right.$
 is regular linear system on $(X,V,U)$ with control operator $B_2$.
Then
$  \left\{
   \begin{array}{ll}
   \dot{z}(t)&=Lz(t)\\
     G_1z(t)&=Kz(t)\\
     G_2z(t)&=v(t)
   \end{array}
 \right.
$ is an abstract linear control system generated by
$(A^I,J^{A,A^I}B_1(I-\overline{K}_{A,B_1})^{-1}\overline{K}_{A,B_2}+J^{A,A^I}B_2).$
If $  \left\{
   \begin{array}{ll}
   \dot{z}(t)&=Lz(t)\\
     G_1z(t)&=0\\
     G_2z(t)&=v(t)
   \end{array}
 \right.
$ is exactly controllable at $t_0$, there exists $k_0>0$ such that $  \left\{
   \begin{array}{ll}
   \dot{z}(t)&=Lz(t)\\
     G_1z(t)&=kKz(t)\\
     G_2z(t)&=v(t)
   \end{array}
 \right.
$ is exactly controllable at $t_0$ whenever $k<k_0.$
\end{theorem}
\begin{proof}\ \
By the assumption, $G_1:D(L)\bigcap Ker{G_2}\rightarrow
U$ and $G_2:D(L)\bigcap Ker{G_1}\rightarrow
V$ are surjectives.
This implies that for any $u\in U$ and $v\in V$,
there exist $z_1\in D(L)\bigcap Ker{G_2}$ and
$z_2\in D(L)\bigcap Ker{G_1}$ such that $G_1z_1=u,\ G_2z_2=v.$
Hence, $z_1+z_2\in D(L)$ and $\left(
           \begin{array}{c}
             G_1 \\
             G_2 \\
           \end{array}
         \right)(z_1+z_2)=\left(
                            \begin{array}{c}
                              u \\
                              v \\
                            \end{array}
                          \right).
$, that is, $\left(
           \begin{array}{c}
             G_1 \\
             G_2 \\
           \end{array}
         \right)$ are surjective.
 Moreover, $B_1=(\lambda-A_{-1})D_{1,\lambda}$ and $B_1=(\lambda-A_{-1})D_{2,\lambda}$ are indicated by the assumption. Here $D_{1,\lambda}u$ and $D_{1,\lambda}u$ are the solution of the equations
 $\left\{
    \begin{array}{ll}
      \lambda z=Lz, & \hbox{} \\
      G_1z=u, & \hbox{} \\
      G_2z=0 & \hbox{}
    \end{array}
  \right.
 $
and $\left\{
    \begin{array}{ll}
      \lambda z=Lz, & \hbox{} \\
      G_1z=0, & \hbox{} \\
      G_2z=v & \hbox{}
    \end{array}
  \right.
 $, respectively.
Hence, the solution $D_\lambda \left(
                                 \begin{array}{c}
                                   u \\
                                   v \\
                                 \end{array}
                               \right)
$ of
$\left\{
    \begin{array}{ll}
      \lambda z=Lz, & \hbox{} \\
      G_1z=u, & \hbox{} \\
      G_2z=v & \hbox{}
    \end{array}
  \right.
 $ satisfies $D_\lambda \left(
                                 \begin{array}{c}
                                   u \\
                                   v \\
                                 \end{array}
                               \right)
=\left(
   \begin{array}{cc}
     D_{1,\lambda} & D_{1,\lambda} \\
   \end{array}
 \right)\left(
                                 \begin{array}{c}
                                   u \\
                                   v \\
                                 \end{array}
                               \right).
$
So the control operator of boundary system
$\left\{
    \begin{array}{ll}
      \dot{z}(t)=Lz & \hbox{} \\
      G_1z=u(t) & \hbox{} \\
      G_2z=v(t) & \hbox{}
    \end{array}
  \right.
 $ is $B=(\lambda-A_{-1})D_\lambda
=(\lambda-A_{-1})
\left(
  \begin{array}{cc}
    D_{1,\lambda} & D_{2,\lambda} \\
  \end{array}
\right)
=\left(
   \begin{array}{cc}
     B_1 & B_2 \\
   \end{array}
 \right).$
Let $Y(t)=\left(
            \begin{array}{c}
              K \\
              0 \\
            \end{array}
          \right)z(t).
$
Since $(A,B_1,K)$ and $(A,B_2,K)$ are regular linear system,
we obtain that $\bigg(A,B,\left(
                             \begin{array}{c}
                               K \\
                               0 \\
                             \end{array}
                           \right)
\bigg)$ is a regular linear system.
Then $\left\{
    \begin{array}{ll}
      \lambda z=Lz, & \hbox{} \\
      G_1z=u, & \hbox{} \\
      G_2z=v & \hbox{}
    \end{array}
  \right.
 $ with output $Y(t)$ is a regular linear system,
the feedthrough operator $D_0$ is computed by
\begin{align*}
 D_0\left(
       \begin{array}{c}
         p \\
         q \\
       \end{array}
     \right)=&
\lim_{\lambda\rightarrow +\infty}\left(
                             \begin{array}{c}
                               K \\
                               0 \\
                             \end{array}
                           \right)(\lambda-A)^{-1}B\left(
       \begin{array}{c}
         p \\
         q \\
       \end{array}
     \right)\\
=&\left(
               \begin{array}{c}
                 \lim_{\lambda\rightarrow +\infty}K(\lambda-A)^{-1}B_1p+\lim_{\lambda\rightarrow +\infty}K(\lambda-A)^{-1}B_2q \\
                 0 \\
               \end{array}
             \right)\\
=& \left(
     \begin{array}{cc}
       \overline{K}_{A,B_1} & \overline{K}_{A,B_2} \\
       0 & 0 \\
     \end{array}
   \right)
\left(
       \begin{array}{c}
         p \\
         q \\
       \end{array}
     \right),\ \forall \left(
       \begin{array}{c}
         p \\
         q \\
       \end{array}
     \right)\in U\times V.
\end{align*}

Sine $I$ is admissible feedback operator of regular linear system $(A,B_1,K,\overline{K}_{A,B_1})$, $I$ is admissible feedback operator for $\bigg(A,B,\left(
                 \begin{array}{c}
                   K \\
                   0 \\
                 \end{array}
               \right),\left(
     \begin{array}{cc}
       \overline{K}_{A,B_1} & \overline{K}_{A,B_2} \\
       0 & 0 \\
     \end{array}
   \right)
\bigg)$. By Lemma \ref{1boundarycontrol}, it follows that
$  \left\{
   \begin{array}{ll}
   \dot{z}(t)&=Lz(t)\\
     G_1z(t)&=Kz(t)+u(t)\\
     G_2z(t)&=v(t)
   \end{array}
 \right.$
is an abstract linear control system with generator
$\bigg(A^I, J^{A,A^I}B\left(
     \begin{array}{cc}
       (I-\overline{K}_{A,B_1})^{-1} & (I-\overline{K}_{A,B_1})^{-1}\overline{K}_{A,B_2} \\
       0 & I \\
     \end{array}
   \right)\bigg)$. Hence $\left\{
   \begin{array}{ll}
   \dot{z}(t)&=Lz(t)\\
     G_1z(t)&=Kz(t)\\
     G_2z(t)&=v(t)
   \end{array}
 \right.$
is an abstract linear control system with generator
$(A^I, J^{A,A^I}B_1(I-\overline{K}_{A,B_1})^{-1}\overline{K}_{A,B_2}+J^{A,A^I}B_2).$
The rest result is obtained directly from Theorem \ref{across}. This completes the proof.
\end{proof}

\begin{example}
We consider Schr$\ddot{o}$dinger equation equation with Dirichlet
boundary control and observation described by
\begin{equation} \label{wave}
\left\{\begin{array}{l}
w_{tt}(x,t)=\Delta w(x,t),\;\; x\in \Omega, t>0, \\
w(x,t)=-\dfrac{\partial(\Delta^{-1}w)}{\partial \nu}, \; \; x\in \Gamma_1,  t\ge 0,\\
w(x,t)=u(x,t), \;\; x\in \Gamma_0,  t\ge 0,
\end{array}\right.
\end{equation}
where $\Omega\subset R^n, n\ge 2$ is an open bounded region with
smooth $C^3$-boundary $\partial
\Omega=\overline{\Gamma_{0}}\cup\overline{\Gamma_{1}}$.
$\Gamma_{0},\Gamma_{1}$ are disjoint parts of the boundary
relatively open in $\partial \Omega$, ${\rm
int}(\Gamma_{1})\neq\emptyset$ and ${\rm
int}(\Gamma_{0})\neq\emptyset$, $\nu$ is the unit normal vector of $\Gamma_0$ pointing towards the exterior of $\Omega$, $u$ is the input function (or control) and $y$ is the output function (or output).

Let $H=H^{-1}(\Omega)$ be the state space and $U=L^2(\partial \Omega)$ be the control (input) or  observation (output) space.
It has been proved in \cite{Guo2005S} that
\begin{align*}
    \left\{\begin{array}{l}
w_{t}(x,t)=-i\Delta w(x,t),\;\; x\in \Omega, t>0, \\
w(x,t)=v(x,t), \; \; x\in \Gamma_1,  t\ge 0,\\
w(x,t)=0, \;\; x\in \Gamma_0,  t\ge 0, \\
 y(x,t)=-i\dfrac{\partial(-\Delta)^{-1}w)}{\partial \nu},\;\; x\in \Gamma_1, t\ge 0
\end{array}\right.
\end{align*}
is regular linear systems with feedthrough operator zero
and $I$ being admissible feedback operator.
Similarly, one can obtain that
\begin{align*}
    \left\{\begin{array}{l}
w_{t}(x,t)=-i\Delta w(x,t),\;\; x\in \Omega, t>0, \\
w(x,t)=0, \; \; x\in \Gamma_1,  t\ge 0,\\
w(x,t)=u(x,t), \;\; x\in \Gamma_0,  t\ge 0, \\
 y(x,t)=-i\dfrac{\partial((-\Delta)^{-1}w)}{\partial \nu},\;\; x\in \Gamma_1, t\ge 0
\end{array}\right.
\end{align*}
is regular linear systems with feedthrough operator zero.
The combinations of \cite{Ammari2002} and \cite{Guo2002L}
implies that system
\begin{align*}
    \left\{\begin{array}{l}
w_{t}(x,t)=-i\Delta w(x,t),\;\; x\in \Omega, t>0, \\
w(x,t)=0, \; \; x\in \Gamma_1,  t\ge 0,\\
w(x,t)=u(x,t), \;\; x\in \Gamma_0,  t\ge 0, \\
 y(x,t)=-i\dfrac{\partial((-\Delta)^{-1}w)}{\partial \nu},\;\; x\in \Gamma_1, t\ge 0
\end{array}\right.
\end{align*}
is exactly controllable at some $t_0>0$.
By Theorem \ref{fedin}, it follows that (\ref{wave})
is an abstract linear control system. Moreover, there exists
a constant $k_0>0$ such that
\begin{align*}
    \left\{\begin{array}{l}
w_{t}(x,t)=-i\Delta w(x,t),\;\; x\in \Omega, t>0, \\
w(x,t)=-ki\dfrac{\partial((-\Delta)^{-1}w)}{\partial \nu}, \; \; x\in \Gamma_1,  t\ge 0,\\
w(x,t)=u(x,t), \;\; x\in \Gamma_0,  t\ge 0,
\end{array}\right.
\end{align*}
is exactly controllable at $t_0>0$ whenever $k<k_0$.
\end{example}

Next, we are concerned with admissible observation and
exactly observation under some regularity perturbation.

\begin{theorem}\label{cross}
Let $(A,B,C,D)$ generate a regular linear system with admissible feedback operator $I$ on $(X,U,U)$, and $(A,B,\Delta C,P)$ generate a regular linear system on $(X,U,Y)$. Then $(A^I,B^I,P(I-D)^{-1}C_\Lambda^A+\Delta C_\Lambda^A)$ generates a regular linear system,
and there holds
$$\Psi_{A^I,P(I-D)^{-1}C_\Lambda^A+\Delta C_\Lambda^A}=F_{A,B,\Delta C,P}(I-F_{A,B,C,D})^{-1}\Psi_{A,C}+\Psi_{A,\Delta C},$$ where $A^I=(A+B(I-D)^{-1}C_\Lambda^A)|_X$ and $B^I=J^{A,A^I}B.$
Moreover, if $(A,\Delta C)$ is exactly observable at $t_0>0$,
then there exists $k_0>0$ such that $(A^I(k),kP(I-kD)^{-1}C_\Lambda^A+\Delta C_\Lambda^A)$ is also exactly observable at $t_0$
whenever $k<k_0,$ where $A^I(k)=(A+B(I-kD)^{-1}kC_\Lambda^A)|_X.$
\end{theorem}

\begin{proof}\ \
Similar to the proof of \cite[Lemma 4.3]{Mei2015}, 
let $\tilde{B}=(B,0)$, $\tilde{C}=\left(
                                   \begin{array}{c}
                                     C \\
                                     \Delta C \\
                                   \end{array}
                                 \right)$, 
$\tilde{D}=\left(
  \begin{array}{cc}
    D & 0 \\
    P & 0 \\
  \end{array}
\right)
$, we obtain that
$(A^I,B^I,P(I-D)^{-1}C_\Lambda^A+\Delta C_\Lambda^A)$ generates a regular linear system and
\begin{align*}
 &\Psi_{A^I,P(I-D)^{-1}C_\Lambda^A+\Delta C_\Lambda^A}\\
=&(0, I)\Psi_{A^I,\tilde{B}^{I_{X\times U}}}\\
=&(0, I)(I-F_{A,\tilde{B},\tilde{C},\tilde{D}})^{-1}\Psi_{A,\tilde{C}}\\
=& (0, I)
\left(
                                       \begin{array}{cc}
                                         (I-F_{A,B,C,D})^{-1} & 0\\
                                         F_{A,B,\Delta C,P} (I-F_{A,B,C,D})^{-1}& I \\
                                       \end{array}
                                     \right)
\left(
  \begin{array}{c}
    \Psi_{A,C} \\
    \Psi_{A,\Delta C} \\
  \end{array}
\right)
\\
=&F_{A,B,\Delta C,P}(I-F_{A,B,C,D})^{-1}\Psi_{A,C}+\Psi_{A,\Delta C}.
\end{align*}

Below we prove the robustness of exactly observability. As stated in the proof of Theorem \ref{across}, $kI$ is admissible feedback for $(A,B,C,D)$ whenever for $k<\frac{1}{\|D\|}$, which indicates that $I$ is admissible feedback for $(A,B,kC,kD)$.
Since $(A,\Delta C)$ is exactly observable at $t_0$, there exists a constant $k_0>0$ such that $\|\Psi_{A,\Delta C}(t_0)x\|\geq k_0\|x\|,\ x\in X$. It follows from the above proof that
\begin{align*}
    &\|\Psi_{A^I,kP(I-kD)^{-1}C_\Lambda^A+\Delta C_\Lambda^A}(t_0)x\|\\
\geq&\|\Psi_{A,\Delta C}(t_0)x\|-\|\Psi_{A^I,kP(I-kD)^{-1}C_\Lambda^A+\Delta C_\Lambda^A}(t_0)x-\Psi_{A,\Delta C}(t_0)x\|\\
\geq& k_0\|x\|-\|F_{A,B,\Delta C,P}(t_0)(I-F_{A,B,kC,kD}(t_0))^{-1}\Psi_{A,kC}(t_0)x\|\\
=& k_0\|x\|-\|F_{A,B,\Delta C,P}(t_0)(I-kF_{A,B,C,D}(t_0))^{-1}k\Psi_{A,C}(t_0)x\|.
\end{align*}
Let $\alpha_0\in (0,k_0)$ and
$$\theta_0=\min\{\frac{1}{\|D\|},\frac{1}{\|F_{A,B,C,D}(t_0)\|},
\frac{k_0-\alpha_0}{(k_0-\alpha_0)\|F_{A,B,C,D}(t_0)\|
+\|F_{A,B,\Delta C,P}(t_0)\|\Psi_{A,C}(t_0)\|\|}\}.$$
Then
$$\|\Psi_{A^I,kP(I-kD)^{-1}C_\Lambda^A+\Delta C_\Lambda^A}(t_0)x\|
> \alpha_0\|x\|,$$
whenever $k<\theta_0$. The proof is therefore completed.
\end{proof}

\begin{remark}
In the special case that $Y=X$ and $B=I$, the above theorem says that both $P$ and $C$ being admissible for $A$ implies that
$C$ is admissible for $A+C$, such result has been proved by Hadd \cite{Hadd2006}.
If $Y=X$ and $C=I$, theorem tells that $(A,B,P)$ generating a regular linear system implies that $((A_{-1}+B)|_X,J^{A,(A_{-1}+B)|_X} B,P^A_\Lambda)$ generates a regular linear system, particularly, $J^{A,(A_{-1}+B)|_X} B$
is admissible for $A+P$.
This means that our result is a generalization of \cite{Hadd2006}.
\end{remark}
\begin{theorem}\label{fedin1}
Assume that the boundary system $(L,G,Q)$ is a regular linear system
generated by $(A,B,G,\overline{G}_{A,B})$ on $(X,U,U)$ with
admissible feedback operator $I$. Suppose that boundary system $(L,G,K)$ is a regular linear
system on $(X,U,Y)$. Then the system
\begin{align}\label{observef}
  \left\{
   \begin{array}{ll}
   \dot{z}(t)&=Lz(t)\\
     Gz(t)&=Qz(t)\\
     y(t)&=Kz(t)
   \end{array}
 \right.
\end{align}
is an abstract linear observation system generated by
$(A^I,K).$
If, in addition, system $
  \left\{
   \begin{array}{ll}
   \dot{z}(t)&=Lz(t)\\
     Gz(t)&=0\\
     y(t)&=Kz(t)
   \end{array}
 \right.$
is exactly observable at some $t_0>0$, there exists a constant
$\theta_0>0$ such that system
$\left\{
   \begin{array}{ll}
   \dot{z}(t)&=Lz(t)\\
     Gz(t)&=kQz(t)\\
     y(t)&=Kz(t)
   \end{array}
 \right.$
is exactly observable at $t_0>0$ whenever $k<\theta_0.$
\end{theorem}
\begin{proof}\ \
Since boundary system $(L,G,K)$ is a regular linear system
with admissible feedback operator $I$, it follows from Lemma \ref{1output} that
\begin{align}\label{K}
    Kz=K_\Lambda^Az+\overline{K}_{A,B}Gz,\ z\in Z,
\end{align}
and
\begin{align}\label{DL}
    Gz=Qz, \ z\in D(A^I)\subset D(L).
\end{align}
The assumption $(L,G,Q)$ is a regular linear system implies
\begin{align}\label{G}
    Qz=Q_\Lambda^Az+\overline{Q}_{A,B}Gz,\ z\in D(L).
\end{align}
Observe that $I-\overline{Q}_{A,B}$ is invertible. The combination of (\ref{DL}) and (\ref{G}) implies that
$Qz=(I-\overline{Q}_{A,B})^{-1}Q_\Lambda^Az,\ z\in D(A^I),$
substituted which into \ref{K} to get
\begin{align*}
    Kz=K_\Lambda^Az+\overline{K}_{A,B}(I-\overline{Q}_{A,B})^{-1}Q_\Lambda^Az,\ z\in D(A^I).
\end{align*}
By Theorem \ref{cross}, (\ref{observef}) is an abstract linear observation system generated by $(A^I,K)$. Furthermore, the rest result is obtained directly from Theorem \ref{cross}. This completes the proof.
\end{proof}

\begin{example}\label{ex}
Consider the following one-dimensional Euler-Bernoulli beam equation
\begin{align}\label{1eb}
    \left\{
      \begin{array}{ll}
        w_{tt}(x,t)+w_{xxxx}(t,x)=0, & \hbox{ }x\in(0,1) \\
        w(0,t)=w_x(0,t)=w_{xx}(1,t)=0,\ w_{xxx}(1,t)=w_t(1,t), & \hbox{ } \\
        y(t)=w_x(1,t). & \hbox{ }
      \end{array}
    \right.
\end{align}

It follows from \cite{Guo2002L} that
 \begin{align}\label{gl}
    \left\{
      \begin{array}{ll}
        w_{tt}(x,t)+w_{xxxx}(t,x)=0, & \hbox{ }x\in(0,1) \\
        w(0,t)=w_x(0,t)=w_{xx}(1,t)=0,\ w_{xxx}(1,t)=u(t), & \hbox{ } \\
        y(t)=w_t(1,t). & \hbox{ }
      \end{array}
    \right.
\end{align}
is a regular linear system with admissible feedback operator $I$ and the corresponding feedback operator is zero.
By Theorem \ref{fedin}, to obtain that (\ref{1eb}) is an abstract linear observation system, we only have to prove that
\begin{align}\label{1ebw}
    \left\{
      \begin{array}{ll}
        w_{tt}(x,t)+w_{xxxx}(t,x)=0, & \hbox{ }x\in(0,1) \\
        w(0,t)=w_x(0,t)=w_{xx}(1,t)=0,\ w_{xxx}(1,t)=u(t), & \hbox{ } \\
        y(t)=w_x(1,t), & \hbox{ }.
      \end{array}
    \right.
\end{align}
is a regular linear systems.
We divide the rest proof into three steps.

\textbf{Step 1.}
Boundary observation system
\begin{align*}
    \left\{
      \begin{array}{ll}
        w_{tt}(x,t)+w_{xxxx}(t,x)=0, & \hbox{ }x\in(0,1) \\
        w(0,t)=w_x(0,t)=w_{xx}(1,t)=0,\ w_{xxx}(1,t)=0, & \hbox{ } \\
        y(t)=w_x(1,t), & \hbox{ }.
      \end{array}
    \right.
\end{align*}
is an abstract linear observation system.
To this end, we let $$F(t)=\frac{1}{2}\int_0^1[w_t^2(x,t)+w_{xx}^2(x,t)]dx.$$
It is not hard to see that $\dot{F}(t)=0$ thereby $F(t)=F(0),\ t\geq 0$. Set $$\rho(t)=\int_0^1x(x-1)w_t(x,t)w_x(x,t)dx.$$
We obtain $|\rho(t)|\leq F(t)=F(0).$ Take the derivative with respect to the time on both sides to get
\begin{align*}
   \dot{\rho}(t)=&\int_0^1x(x-1)w_{tt}(x,t)w_x(x,t)dx+
\int_0^1x(x-1)w_{t}(x,t)w_{xt}(x,t)dx\\
 =&-\int_0^1x(x-1)w_{xxxx}(x,t)w_x(x,t)dx+
\frac{1}{2}\int_0^1x(x-1)\frac{\partial}{\partial x}w^2_{t}(x,t)dx\\
=&-x(x-1)w_{xxx}(x,t)w_x(x,t)|^1_{x=0}+\int_0^1x(x-1)w_{xxx}(x,t)w_{xx}(x,t)dx\\
&+\int_0^1(2x-1)w_{xxx}(x,t)w_{x}(x,t)dx+\frac{1}{2}\int_0^1x(x-1)\frac{\partial}{\partial x}w^2_{t}(x,t)dx\\
=&\frac{1}{2}\int_0^1x(x-1)\frac{\partial}{\partial x}w^2_{xx}(x,t)dx+\frac{1}{2}\int_0^1x(x-1)\frac{\partial}{\partial x}w^2_{t}(x,t)dx\\
&+(2x-1)w_{xx}(x,t)w_{x}(x,t)|_{x=0}^1
-\int_0^1w_{xx}(x,t)[2w_{x}(x,t)+(2x-1)w_{xx}(x,t)]dx\\
=&\frac{1}{2}x(x-1)[w^2_{t}(x,t)+w^2_{xx}(x,t)]|^1_{x=0}
-\frac{1}{2}\int_0^1(2x-1)[w^2_{t}(x,t)+3w^2_{xx}(x,t)]dx\\
&-\int_0^12w_{xx}(x,t)w_{x}(x,t)dx\\
=&-\frac{1}{2}\int_0^1(2x-1)[w^2_{t}(x,t)+3w^2_{xx}(x,t)]dx
-w^2_{x}(1,t).
\end{align*}
Integrate from $0$ to $T$ with respect to $t$ to derive
\begin{align*}
  \int_0^Tw^2_{x}(1,t)dt&=\int_0^T\bigg[
-\frac{1}{2}\int_0^1(2x-1)[w^2_{t}(x,t)+3w^2_{xx}(x,t)]dx\bigg]dt
+ \rho(0)-\rho(T) \\
&\leq (3T+2)F(0).
\end{align*}

\textbf{Step 2.}
Boundary system (\ref{1ebw}) is a well-posed. We consider the boundary system under the zero initial condition:
$w(x,0)=w_t(x,0)=0$.
Define $F(t)$ and $\rho$ as the same in Step 1. By \cite{Guo2008W}, it follows that
$$F(t)\leq C_{\delta, T}\int_0^Tu^2(t)dt,\ \forall t\in [0,T],$$
where $\delta\in(0,\frac{1}{1+4T})$ and
$C_{\delta, T}=\frac{1+\delta+4T}{2[1-(1+4T)\delta]}+\frac{1}{2\delta}.$

Observe that $|\rho(t)|\leq F(t)$ holds. Take the derivative with respect to the time on both sides to get
\begin{align*}
   \dot{\rho}(t)=-\frac{1}{2}\int_0^1(2x-1)[w^2_{t}(x,t)+3w^2_{xx}(x,t)]dx
-w^2_{x}(1,t).
\end{align*}
Integrate from $0$ to $T$ with respect to $t$ to derive
\begin{align*}
  \int_0^Tw^2_{x}(1,t)dt&=\int_0^T\bigg[
-\frac{1}{2}\int_0^1(2x-1)[w^2_{t}(x,t)+3w^2_{xx}(x,t)]dx\bigg]dt
-\rho(T) \\
&\leq 3\int_0^TF(t)dt+F(T)\\
&\leq(1+3T)C_{\delta,T}\int_0^Tu^2(t)dt.
\end{align*}

\textbf{Step 3.}
Boundary system (\ref{1ebw}) is regular.
Denote by $\hat{w}(x,s)$ the Laplace transform of $w(x,s)$ with respect to $t$, that is, $\hat{w}(x,s)=\int_0^\infty w(x,t)e^{-st}dt.$ Similarly, the Laplace transform $\hat{u}(s)$ of $u(t)$ with respect to $t$ is $\hat{u}(s)=\int_0^\infty u(t)e^{-st}dt.$
For the zero initial condition $w(x,0)=w_t(x,0)=0$, we get
\begin{align*}
    \left\{
      \begin{array}{ll}
        s^2\hat{w}(x,s)+\hat{w}_{xxxx}(x,s)=0, & \hbox{ }x\in(0,1) \\
        \hat{w}(0,s)=\hat{w}_x(0,s)=\hat{w}_{xx}(1,s)=0,\ \hat{w}_{xxx}(1,s)=\hat{u}(s), & \hbox{ } \\
        \hat{y}(s)=\hat{w}_x(1,t), & \hbox{ }.
      \end{array}
    \right.
\end{align*}
Denote by $H(s)$ the corresponding transform function.
Since the system is well-posed. Then we have that $H(s)$ satisfies that $\hat{y}(s)=H(s)\hat{u}(s)$ and it is bounded on some right half plane. In order derive the regularity, we only need to show that the limit of transfer function exists as $s\rightarrow +\infty.$ So we can set $s>0$ and $t=\sqrt{\frac{s}{2}}$.
The first equation implies that
$$\hat{w}(x,s)=ach(tx)cos(tx)+bch(tx)sin(tx)+csh(tx)cos(tx)
+dsh(tx)sin(tx)$$
with $a,b,c$ and $d$ being to be determined.
Use $\hat{w}(0,s)=0$ to get $a=0$.
We obtain
\begin{align*}
    \hat{w}_x(x,s)=&t[bsh(tx)sin(tx)+(b+c)ch(tx)cos(tx)
-csh(tx)sin(tx)\\
&+dch(tx)sin(tx)+dsh(tx)cos(tx)].
\end{align*}
Use $\hat{w}_x(0,s)=0$ to get $b+c=0$.
We obtain
\begin{align*}
    \hat{w}_{xx}(x,s)=&2t^2[bch(tx)sin(tx)
+bsh(tx)cos(tx)+dch(tx)cos(tx)]
\end{align*}
and
\begin{align*}
    \hat{w}_{xxx}(x,s)=&2t^3[2bch(tx)cos(tx)
+dsh(tx)cos(tx)-dch(tx)sin(tx)].
\end{align*}
Use $\hat{w}_{xx}(1,s)=0$ and $\hat{w}_{xxx}(1,s)=\hat{u}(s)$ to get
\begin{align*}
    \left\{
       \begin{array}{ll}
         b=\frac{chtcost}{2t^3(ch^2t+cos^2t)}\hat{u}(s), & \hbox{} \\
         d=-\frac{chtsint+shtcost}{2t^3(ch^2t+cos^2t)}\hat{u}(s). & \hbox{}
       \end{array}
     \right.
\end{align*}
Hence we obtain that
$$H(s)=-\frac{shtchtsintcost-[chtsint+shtcost]^2}{2t^2(ch^2t+cos^2t)}.$$
Observe that
\begin{align*}
    |H(s)|\leq \frac{ch^2t+(2cht)^2}{2t^2ch^2t}=\frac{5}{2t^2}=\frac{5}{s}.
\end{align*}
Hence $H(s)\rightarrow 0$ as $s\rightarrow +\infty$. The regularity of (\ref{1ebw}) is therefore proved. This completes the proof.
\end{example}

\begin{example}
Consider one-dimensional Euler-Bernoulli beam equation
\begin{align}\label{1eb1}
    \left\{
      \begin{array}{ll}
        w_{tt}(x,t)+w_{xxxx}(t,x)=0, & \hbox{ }x\in(0,1) \\
        w(0,t)=w_x(0,t)=w_{xx}(1,t)=0,\ w_{xxx}(1,t)=w_t(1,t), & \hbox{ } \\
        y(t)=w_{xx}(0,t). & \hbox{ }
      \end{array}
    \right.
\end{align}
Guo, Wang and Yang showed in \cite{Guo2008W} that
 \begin{align}\label{lebc}
    \left\{
      \begin{array}{ll}
        w_{tt}(x,t)+w_{xxxx}(t,x)=0, & \hbox{ }x\in(0,1) \\
        w(0,t)=w_x(0,t)=w_{xx}(1,t)=0,\ w_{xxx}(1,t)=w(1,t), & \hbox{ } \\
        y(t)=w_{xx}(0,t). & \hbox{ }
      \end{array}
    \right.
\end{align}
is a well-posed linear system.
From Step 3 of the above Example \ref{ex}, we obtain that
the transform function of system (\ref{lebc}) is presented
on $R^+$ by
\begin{align*}
   H_1(s)=\frac{\hat{w}_{xx}(0,s)}{\hat{u}(s)}=\frac{2t^2d}{\hat{u}(s)}
=-\frac{chtsint+shtcost}{t(ch^2t+cos^2t)}
\end{align*}
with $t=\sqrt{\frac{s}{2}}$.
Then
\begin{align*}
   |H_1(s)| \leq \frac{cht+sht}{tch^2t}\leq \frac{2}{tcht}\rightarrow 0,
\end{align*}
as $s\rightarrow +\infty,$ that is,
system (\ref{lebc}) is regular. Observe that system (\ref{gl})
is regular. By Theorem \ref{fedin1}, system (\ref{1eb1}) is an abstract linear observation system.

Next, we show that system
\begin{align}\label{1eb11}
    \left\{
      \begin{array}{ll}
        w_{tt}(x,t)+w_{xxxx}(t,x)=0, & \hbox{ }x\in(0,1) \\
        w(0,t)=w_x(0,t)=w_{xx}(1,t)=0,\ w_{xxx}(1,t)=0, & \hbox{ } \\
        y(t)=w_{xx}(0,t), & \hbox{ }
      \end{array}
    \right.
\end{align}
is exactly observable.
Let $$F(t)=\frac{1}{2}\int_0^1[w_t^2(x,t)+w_{xx}^2(x,t)]dx.$$
We have $\dot{F}(t)=0$ thereby $F(t)=F(0),\ t\geq 0$. Set $$\rho_1(t)=\int_0^1(x-1)w_t(x,t)w_x(x,t)dx.$$
Obviously, $|\rho_1(t)|\leq F(t)=F(0).$
We compute
$$\dot{\rho}_1(t)=\frac{1}{2}w_{xx}^2(0,t)-
\frac{1}{2}\int_0^1(w^2_t(x,t)+3w^2_{xx}(x,t))dx.$$
Integrate from $0$ to $T$ with respect to $t$ to get
\begin{align*}
\int_0^Tw_{xx}^2(0,t)dt=&
\frac{1}{2}\int_0^t\int_0^1(w^2_t(x,t)+3w^2_{xx}(x,t))dxdt+\rho_1(0)
-\rho_1(T)\\
\geq & (T-2)E(0).
\end{align*}
This indicates that system (\ref{1eb11}) is exactly observable at any $T>2$.
Then, by Theorem \ref{fedin1}, system
\begin{align*}
    \left\{
      \begin{array}{ll}
        w_{tt}(x,t)+w_{xxxx}(t,x)=0, & \hbox{ }x\in(0,1) \\
        w(0,t)=w_x(0,t)=w_{xx}(1,t)=0,\ w_{xxx}(1,t)=w_t(1,t), & \hbox{ } \\
        y(t)=w_{xx}(0,t). & \hbox{ }
      \end{array}
    \right.
\end{align*}
is exactly observable at $T>2$ whenever $k$ is small enough.
\end{example}

In the rest of this section, we
are concern with the regularity under perturbations.

\begin{theorem}\label{bcross}
Assume that $(A,B,C,D)$ generates a regular linear system with admissible feedback operator $I$. Suppose $(A,\Delta B,C)$, $(A,B,\Delta C)$ and $(A,\Delta B,\Delta C)$ generates regular linear systems. Then $(A^I,J^{A,A^I}\Delta B,\Delta C_\Lambda^A)$ generates a regular linear system.
\end{theorem}
\begin{proof}\ \
By Theorem \ref{across}, it follows that $(A^I,J^{A,A^I}\Delta B)$ generates an abstract linear control system with
$$\Phi_{A^I,J^{A,A^I}\Delta B}=\Phi_{A,B}(I-F_{A,B,C,D})^{-1}F_{A,\Delta B,C}+\Phi_{A,\Delta B}.$$
Theorem \ref{cross} implies that $(A^I,\Delta C^A_\Lambda)$ generates an abstract linear observation system with
$$\Psi_{A^I,\Delta C_\Lambda^A}=F_{A,B,\Delta C}(I-F_{A,B,C,D})^{-1}\Psi_{A,C}+\Psi_{A,\Delta C}.$$
Since $(A,B,\Delta C)$ and $(A,\Delta B,\Delta C)$ are
regular linear system, we define
$F=F_{A,B,\Delta C}(I-F_{A,B,C,D})^{-1}F_{A,\Delta B,C}+F_{A,\Delta B,\Delta C}$.
Then it is not hard to verify that $(T_{A^I},\Phi_{A^I,J^{A,A^I}\Delta B},\Psi_{A^I,\Delta C_\Lambda^A},F)$ is a regular linear system
generated by $(A^I,J^{A,A^I}\Delta B,\Delta C_\Lambda^A)$.
The proof is therefore completed.
\end{proof}

\begin{remark}
In the special case that $Y=X$ and $C=I$, the above theorem says that both $(A,B,\Delta C)$ and $(A,\Delta B,\Delta C)$ being regular linear system implies that
$((A_{-1}+B)|_X,J^{A,(A_{-1}+B)|_X}\Delta B,\Delta C^A_\Lambda)$ generates a regular linear system, such result has been proved by Hadd \cite{Hadd2006};
If $Y=X$ and $B=I$, the above theorem says that both $(A,\Delta B,C)$ and $(A,\Delta B,\Delta C)$ being regular linear system implies that
$(A+C,J^{A,+C}\Delta B,\Delta C)$ generates a regular linear system.
This means that our result is a generalization of \cite{Hadd2006}.
\end{remark}
\begin{theorem}\label{fedin2}
Assume that the boundary system
$  \left\{
   \begin{array}{ll}
   \dot{z}(t)&=Lz(t)\\
     G_1z(t)&=u(t)\\
     G_2z(t)&=0\\
     y(t)&=Kz(t)
   \end{array}
 \right.$
is regular linear system generated by $(A,B_1,K,\overline{K}_{A,B_1})$ with $I$ being admissible feedback operator. Suppose that
$
  \left\{
   \begin{array}{ll}
   \dot{z}(t)&=Lz(t)\\
     G_1z(t)&=u(t)\\
     G_2z(t)&=0\\
     y(t)&=W z(t)
   \end{array}
 \right.$ and
 $  \left\{
   \begin{array}{ll}
   \dot{z}(t)&=Lz(t)\\
     G_1z(t)&=0\\
     G_2z(t)&=v(t)\\
     y(t)&=\left(
             \begin{array}{c}
               K \\
               W \\
             \end{array}
           \right)     z(t)
   \end{array}
 \right.$
 are regular linear systems with $B_2$ being the control operator of the second system.
Then
\begin{align}\label{mix}
  \left\{
   \begin{array}{ll}
   \dot{z}(t)&=Lz(t)\\
     G_1z(t)&=Kz(t)\\
     G_2z(t)&=v(t)\\
     y(t)&=Wz(t)
   \end{array}
 \right.
\end{align}
 is a regular linear system generated by
$(A^I,J^{A,A^I}B_1(I-\overline{K}_{A,B_1})^{-1}\overline{K}_{A,B_2}+J^{A,A^I}B_2,
W,\overline{W}_{A,B_1}(I-\overline{K}_{A,B_1})^{-1}\overline{K}_{A,B_2}
+\overline{W}_{A,B_2}).$
\end{theorem}
\begin{proof}\ \
By Theorem \ref{fedin}, it follows that $  \left\{
   \begin{array}{ll}
   \dot{z}(t)&=Lz(t)\\
     G_1z(t)&=Kz(t)\\
     G_2z(t)&=v(t)
   \end{array}
 \right.
$ is an abstract linear control system with generator
$(A^I, J^{A,A^I}B_1(I-\overline{K}_{A,B_1})^{-1}\overline{K}_{A,B_2}+J^{A,A^I}B_2)$. It follows from Theorem \ref{fedin1} that
$  \left\{
   \begin{array}{ll}
   \dot{z}(t)&=Lz(t)\\
     G_1z(t)&=Kz(t)\\
     G_2z(t)&=0\\
     y(t)&=Wz(t)
   \end{array}
 \right.$ is an abstract linear observation system with generator
$(A^I, W)$
and the restriction of $W$ to
$D(A^I)$ is equal to $W_\Lambda^A
+\overline{W}_{A,B_1}(I-\overline{K}_{A,B_1})^{-1}K_\Lambda^A$.
By Theorem \ref{bcross}, our assumptions imply that
$(A^I,J^{A,A^I}B_2,W_\Lambda^A)$ generates a regular linear system. Combining this with the boundedness of operator $J^{A,A^I}B_1(I-\overline{K}_{A,B_1})^{-1}\overline{K}_{A,B_2}$ implies that $(A^I,J^{A,A^I}B_1(I-\overline{K}_{A,B_1})^{-1}\overline{K}_{A,B_2}+J^{A,A^I}B_2,
W_\Lambda^A)$ generates a regular linear system.
By Theorem \ref{across}, we obtain that $(A^I,J^{A,A^I}B_1(I-\overline{K}_{A,B_1})^{-1}\overline{K}_{A,B_2}+J^{A,A^I}B_2,
(I-\overline{K}_{A,B_1})^{-1}K_\Lambda^A)$ generates a regular linear system. Since $\overline{W}_{A,B_1}$ is bounded,
$(A^I,J^{A,A^I}B_1(I-\overline{K}_{A,B_1})^{-1}\overline{K}_{A,B_2}+J^{A,A^I}B_2,
\overline{W}_{A,B_1}(I-\overline{K}_{A,B_1})^{-1}K_\Lambda^A)$ generates a regular linear system.
Therefore, $(A^I,J^{A,A^I}B_1(I-\overline{K}_{A,B_1})^{-1}\overline{K}_{A,B_2}+J^{A,A^I}B_2,
W)$ is a regular linear system. Hence the regularity of system (\ref{mix}) is obtained by Lemma \ref{1output}.

Next, we shall compute the feedthrough operator.
By Theorem \ref{across}, for any enough big $ Re(\lambda)$, we have
\begin{align*}
   &(\lambda-(A^I)_{-1})^{-1}\bigg(J^{A,A^I}B_1(I-\overline{K}_{A,B_1})^{-1}\overline{K}_{A,B_2}
+J^{A,A^I}B_2\bigg)\\
=& (\lambda-A_{-1})^{-1}B_1(I-G_{A,B_1,K,\overline{K}_{A,B_1}}(\lambda))^{-1}
G_{A,B_2,K,\overline{K}_{A,B_2}}(\lambda)+(\lambda-A_{-1})^{-1}B_2,
\end{align*}
and the transform function of (\ref{mix}) is given by
\begin{align*}
   W(\lambda-A_{-1})^{-1}B_1(I-G_{A,B_1,K,\overline{K}_{A,B_1}}(\lambda))^{-1}
G_{A,B_2,K,\overline{K}_{A,B_2}}(\lambda)+W(\lambda-A_{-1})^{-1}B_2.
\end{align*}
Observe that the assumption implies that the strong limit
\begin{align*}
   &\lim_{\lambda\rightarrow +\infty}\bigg(W(\lambda-A_{-1})^{-1}B_1(I-G_{A,B_1,K,\overline{K}_{A,B_1}}(\lambda))^{-1}
G_{A,B_2,K,\overline{K}_{A,B_2}}(\lambda)+W(\lambda-A_{-1})^{-1}B_2\bigg)\\
=&\lim_{\lambda\rightarrow +\infty}W(\lambda-A_{-1})^{-1}B_1(I-G_{A,B_1,K,\overline{K}_{A,B_1}}(\lambda))^{-1}
G_{A,B_2,K,\overline{K}_{A,B_2}}(\lambda)\\
&+\lim_{\lambda\rightarrow +\infty}W(\lambda-A_{-1})^{-1}B_2\\
=&\bigg(\overline{W}_{A,B_1}(I-\overline{K}_{A,B_1})^{-1}\overline{K}_{A,B_2}
+\overline{W}_{A,B_2}\bigg)
\end{align*}
hold.
Therefore the feedthrough operator of (\ref{mix}) is $\overline{W}_{A,B_1}(I-\overline{K}_{A,B_1})^{-1}\overline{K}_{A,B_2}
+\overline{W}_{A,B_2}$.
The proof is therefore completed.
\end{proof}

\begin{example}
Consider the following boundary system governed by wave equations
\begin{equation} \label{1.1}
\left\{\begin{array}{l}
w_{tt}(x,t)=\Delta w(x,t),\;\; x\in \Omega, t>0, \\
w(x,t)=-\dfrac{\partial(\mathcal{A}^{-1}w)}{\partial \nu}, \; \; x\in \Gamma_1,  t\ge 0,\\
w(x,t)=u(x,t), \;\; x\in \Gamma_0,  t\ge 0, \\
 y(x,t)=-\dfrac{\partial(\mathcal{A}^{-1}w)}{\partial \nu},\;\; x\in \Gamma_0, t\ge 0
\end{array}\right.
\end{equation}
where $\Omega\subset R^n, n\ge 2$ is an open bounded region with
smooth $C^3$-boundary $\partial
\Omega=\overline{\Gamma_{0}}\cup\overline{\Gamma_{1}}$.
$\Gamma_{0},\Gamma_{1}$ are disjoint parts of the boundary
relatively open in $\partial \Omega$, ${\rm
int}(\Gamma_{1})\neq\emptyset$ and ${\rm
int}(\Gamma_{0})\neq\emptyset$, $\nu$ is the unit normal vector of $\Gamma_0$ pointing towards the exterior of $\Omega$, $u$ is the input function (or control) and $y$ is the output function (or output).

Let $H=L^2(\Omega)\times H^{-1}(\Omega)$ be the state space and $U=L^2(\partial \Gamma_0), \ V=L^2(\partial \Gamma_1)$ be the control (input) or  observation (output) space.
Guo and Zhang \cite{Guo2005Z} proved that system
\begin{equation*}
\left\{\begin{array}{l}
w_{tt}(x,t)=\Delta w(x,t),\;\; x\in \Omega, t>0, \\
w(x,t)=v(x,t), \; \; x\in \Gamma_1,  t\ge 0,\\
w(x,t)=0, \;\; x\in \Gamma_0,  t\ge 0, \\
 y(x,t)=-\dfrac{\partial(\mathcal{A}^{-1}w)}{\partial \nu},\;\; x\in \Gamma_1, t\ge 0
\end{array}\right.
\end{equation*}
is a regular linear system with feedthrough operator $I$ and with admissible feedback operator $I$.
Moreover,
\begin{equation*}
\left\{\begin{array}{l}
w_{tt}(x,t)=\Delta w(x,t),\;\; x\in \Omega, t>0, \\
w(x,t)=0, \; \; x\in \Gamma_1,  t\ge 0,\\
w(x,t)=u(x,t), \;\; x\in \Gamma_0,  t\ge 0, \\
 y(x,t)=-\dfrac{\partial(\mathcal{A}^{-1}w)}{\partial \nu},\;\; x\in \Gamma_1, t\ge 0
\end{array}\right.
\end{equation*}
is a regular linear system.
By the same procedure, one can verify that
\begin{equation*}
\left\{\begin{array}{l}
w_{tt}(x,t)=\Delta w(x,t),\;\; x\in \Omega, t>0, \\
w(x,t)=v(x,t), \; \; x\in \Gamma_1,  t\ge 0,\\
w(x,t)=0, \;\; x\in \Gamma_0,  t\ge 0, \\
 y(x,t)=-\dfrac{\partial(\mathcal{A}^{-1}w)}{\partial \nu},\;\; x\in \Gamma_0, t\ge 0
\end{array}\right.
\end{equation*}
and
\begin{equation*}
\left\{\begin{array}{l}
w_{tt}(x,t)=\Delta w(x,t),\;\; x\in \Omega, t>0, \\
w(x,t)=0, \; \; x\in \Gamma_1,  t\ge 0,\\
w(x,t)=u(x,t), \;\; x\in \Gamma_0,  t\ge 0, \\
 y(x,t)=-\dfrac{\partial(\mathcal{A}^{-1}w)}{\partial \nu},\;\; x\in \Gamma_1, t\ge 0
\end{array}\right.
\end{equation*}
are regular linear systems. Then we claim by Theorem \ref{fedin2} that system (\ref{1.1}) is regular with feedthrough operator $I$.
\end{example}


\end{document}